\documentclass[final,1p,times]{elsarticle}
\usepackage[colorlinks=true]{hyperref}
\usepackage{mathrsfs}
\usepackage{amsthm,amsmath,amssymb,mathrsfs,amsfonts,graphicx,graphics,latexsym,exscale,cmmib57,dsfont,bbm,amscd,ulem}
\usepackage{color,soul}%%%\hl{ highlighting text } % use hl{ } to highlight the sentence.
\newtheorem{Athm}{Theorem}
\newtheorem{Alem}[Athm]{Lemma}

\theoremstyle{definition}

\newtheorem*{rem}{Note}
\newtheorem{Adefn}[Athm]{Definition}

\theoremstyle{remark}

\journal{Semigroup Forum (Nov. 24, 2015; accepted Jun. 8, 2016)}

\begin{document}

\begin{frontmatter}
\title{Dynamics of Bohr almost periodic motions of topological abelian semigroups}
%\author[K.~Cao]{Kai Cao}
%\address{Department of Mathematics, Nanjing University,
%Nanjing 210093\\
%People's Republic of China}
%\email{450297450@qq.com}
\author{Xiongping Dai}
\ead{xpdai@nju.edu.cn}
%%\cortext[cor1]{Corresponding author}

\address{Department of Mathematics, Nanjing University, Nanjing 210093, People's Republic of China}

\begin{abstract}
We study the topological and ergodic dynamics of Bohr almost periodic motions of a topological abelian semigroup acting continuously on a compact metric space.
\end{abstract}

\end{frontmatter}

%%% ----------------------------------------------------------------------
%%% ----------------------------------------------------------------------
%\tableofcontents
%%%%%%%%%%%%%%%%%%%%%%%%%%%%%%%%%%%%%%%%%%
%%%%%%%%%%%%%%%%%%%%%%%%%%%%%%%%%%%%%%%%%%

\section{Introduction}\label{sec1}%%%
We will consider, in this paper, the topological and ergodic dynamics of Bohr almost periodic motions of topological abelian semigroups acting continuously on a compact metric space $(X,d)$.

Let $f\colon\mathbb{Z}_+\times X\rightarrow X;\ (t,x)\mapsto f(t,x)$ be a discrete-time semi flow on the space $X$.
Recall that a point $x\in X$ or a motion $f(\centerdot,x)$ is called \textit{almost periodic of Bohr} if and only if
\begin{itemize}
\item for any $\varepsilon>0$ there exists a relatively dense subset $\{\tau_n\}$ of $\mathbb{Z}_+$ which possesses the following property:
\begin{equation*}d(f(t,x),f(t+\tau_n,x))<\varepsilon\quad\forall t\in\mathbb{Z}_+.\end{equation*} 
Here a subset $S$ of $\mathbb{Z}_+$ is called `relatively dense' if one can find an integer $L>0$ such that $S\cap[n,n+L)\not=\varnothing\ \forall n\in\mathbb{Z}_+$; cf.~\cite[Definition~V7.08]{NS}.
\end{itemize}
This is equivalent to say that $f_x\colon\mathbb{Z}_+\rightarrow X$, defined by $t\mapsto f(t,x)$, is a Bohr almost periodic function (Bohr 1925, 1926).
See \cite[Definition~V8.01]{NS} for the $\mathbb{R}$-action system case.

By a topological semigroup, like $\mathbb{R}_+^d, \mathbb{Z}_+^d$ and $\mathbb{N}^d$, we mean a semigroup endowed with a Hausdorff topology that renders the algebraic  operations sum $+$ or product $\circ$ continuous.
We now will extend Bohr almost periodic motion to topological semigroup acting continuously on the compact metric space $X$.

From here on out, assume $G$ is a topological semigroup with a product operation $\circ$ and an identity $e$, which acts continuously from left on the compact metric space $X$,
\begin{gather*}
\pi\colon G\times X\rightarrow X,
\end{gather*}
simply written as
\begin{equation*}
G\curvearrowright X;\quad (g,x)\mapsto g(x)=\pi(g,x).
\end{equation*}
Here are the basic notations we study in this paper.

\begin{Adefn}\label{Adef1}%%%
Let $(G,\circ)$ be a topological semigroup with an identity element $e$.
\begin{enumerate}
\item[(a)] A subset $T$ of $G$ is called \textit{left syndetic} in $G$ if there exists a compact subset $L$ of $G$ such that $T\circ L=G$ (cf.~\cite[Definition~2.02]{GH}).

\item[(b)] Further a point $x\in X$ is called \textit{Bohr almost periodic for $G\curvearrowright X$} if for any $\varepsilon>0$ the $\varepsilon$-error period set
\begin{equation*}
\mathbb{P}(\varepsilon)=\left\{\tau\in G\,|\, d(g(x),\tau\circ g(x))<\varepsilon\ \forall g\in G\right\}
\end{equation*}
is left syndetic in $G$ under the sense of (a).
\end{enumerate}

Clearly, if $x$ is Bohr almost periodic for $G\curvearrowright X$, then each point of the orbit $G(x)$ is also Bohr almost periodic for $G\curvearrowright X$. Thus we may say $G(x)$ is Bohr almost periodic and moreover the subsystem $G\curvearrowright\textrm{cls}_X{G(x)}$ is von Neumann almost periodic (cf.~\cite[Remark~4.32]{GH} and \cite{ChD}).
\end{Adefn}

Since every syndetic subset of $\mathbb{Z}_+$ is also relatively dense, Def.~\ref{Adef1}(b) is a generalization of the classical Bohr almost periodic motion. However, for a $\mathbb{Z}_+$-action, generally an almost periodic motion in the sense of \cite{NS} is not necessarily to be Bohr almost periodic in the sense of Def.~\ref{Adef1}(b); this is because a relatively dense subset $S$ of $\mathbb{Z}_+$ like $S=3\mathbb{N}$ does not need to be syndetic in the sense of Def.~\ref{Adef1}(a).
See, e.g.,~\cite{ChD} for some comparison with uniform recurrence and von Neumann almost periodicity. Particularly, since $G$ is not necessarily a group here, we even cannot sure the orbit closure $\textrm{cls}_X{G(x)}$ of a Bohr almost periodic point $x$ for $G\curvearrowright X$ is minimal when $G\not=\mathbb{Z}_+^d$ and $\not=\mathbb{R}_+^d$ (cf.~\cite[Corollary~2.7]{ChD}).

It should be noted here that the $\varepsilon$-error period set $\mathbb{P}(\varepsilon)$ in Def.~\ref{Adef1}(b) is not required to be a \textit{sub-semigroup} of $G$. Otherwise it is named ``uniform regular almost periodicity'' and the latter is systematically studied in \cite{Eg, MR} when $G$ is assumed to be a topological group.

In this paper, we shall consider the topological structure (Theorem~\ref{Athm4} in $\S\ref{sec2}$), the probabilistic structure (Theorem~\ref{Athm6} in $\S\ref{sec3}$), and the pointwise ergodic behavior (Theorem~\ref{Athm9} in $\S\ref{sec4}$), of Bohr almost periodic motions of topological abelian semigroups acting continuously on the compact metric space $X$.
The theory of Bohr almost periodic motions of abelian semigroup herself has some essential difficulties comparing with the abelian groups. For example, this theory will involve the Haar measure of a locally compact second countable abelian semigroup; but such a semigroup does not need to have a classical Haar measure like $G=\mathbb{R}_+$ and $\mathbb{Z}_+$, since the Lebesgue and the counting measures are not translation invariant; for example, for $L\colon x\mapsto x+1$ from $\mathbb{Z}_+$ to $\mathbb{Z}_+$, $3=\#\{0,1,2\}\not=\#L^{-1}\{0,1,2\}=2$.

%%%%%%%%%%%%%%%%%%%%%%%%%%%%%%%%%%%
\section{Topological structures}\label{sec2}%%%
The following theorem establishes the equicontinuity of Bohr almost periodic motions of a topological semigroups $G$ acting continuously on a compact metrizable space $X$, which is a generalization of A.A. Markov's Lyapunov stability theorem
of continuous-time flows with $G=\mathbb{R}$ (cf.~\cite[Theorem~V8.05]{NS}).

\begin{Athm}\label{Athm2}%%5
Let $y\in X$ be Bohr almost periodic for $G\curvearrowright X$. Then the family of transformations
\begin{gather*}
\left\{g\colon\mathrm{cls}_X{G(y)}\rightarrow\mathrm{cls}_X{G(y)};\ x\mapsto g(x)\right\}_{g\in G}
\end{gather*}
is equicontinuous; that is to say, 
\begin{itemize}
\item 
for any $\varepsilon>0$ there exists a $\delta>0$ such that if $x,z\in\mathrm{cls}_X{G(y)}$ with $d(x,z)<\delta$ then $d(g(x),g(z))<\varepsilon$ for every $g\in G$.
\end{itemize}
\end{Athm}
\begin{rem}
See~\cite[Theorem~4.37]{GH} for the group case under the guise that $G\curvearrowright X$ is (Bohr) almost periodic. If the compactness of the underlying space $X$ is weakened by the uniform continuity of the transformations $g\colon X\rightarrow X$ for all $g\in G$ (cf.~\cite[Definition~4.36]{GH}) but $G$ is discrete there, then the statement of this theorem still holds by the same argument below.
\end{rem}

\begin{proof}
Let $\varepsilon>0$ be arbitrarily given. Then because of the Bohr almost periodicity of the point $y$ for $G\curvearrowright X$ there exists a compact subset $L=L(\varepsilon)$ of $G$ for which the $\frac{\varepsilon}{3}$-error period set
$$
\mathbb{P}(\varepsilon/3)=\left\{\tau\in G\,|\,d(g(y),\tau\circ g(y))<\frac{\varepsilon}{3}\ \forall g\in G\right\}
$$
is such that $\mathbb{P}(\varepsilon/3)\circ L=G$.
Since $\textrm{cls}_X{G(y)}$ and $L$ both are compact and $G$ acts continuously on $X$, there exists a number $\delta>0$ such that for any two points $x,z\in\textrm{cls}_X{G(y)}$ it will follow from $d(x,z)<\delta$ that
$$
d(\ell(x),\ell(z))<\frac{\varepsilon}{3}\quad\forall \ell\in L.
$$
To prove Theorem~\ref{Athm2}, it is sufficient to show that: if $d(t_1(y),t_2(y))<\delta$ for two elements $t_1,t_2\in G$, then $d(g\circ t_1(y),g\circ t_2(y))<\varepsilon$ for all $g\in G$.

For that, we now choose an arbitrary $g\in G$ and we then can pick two elements $\ell\in L$ and $\tau\in\mathbb{P}(\varepsilon/3)$ such that $g=\tau\circ\ell$. Hence
\begin{equation*}\begin{split}
d(g\circ t_1(y),g\circ t_2(y))&=d(\tau\circ\ell\circ t_1(y),\tau\circ\ell\circ t_2(y))\\
&<d(\ell\circ t_1(y),\ell\circ t_2(y))+\frac{2\varepsilon}{3}\\
&<\varepsilon.
\end{split}\end{equation*}
This concludes the proof of Theorem~\ref{Athm2}.
\end{proof}

For our further result, we need a lemma in which we assume $G$ is commutative.

\begin{Alem}\label{Alem3}%%%
Let $y\in X$ be a Bohr almost periodic point for $G\curvearrowright X$ and $t_n,s_n\in G$, where $G$ is a topological abelian semigroup. Then if $\{t_n(y)\}_1^\infty$ and $\{s_n(y)\}_1^\infty$ are two Cauchy sequences in $X$, the sequence $\{t_n\circ s_n(y)\}_1^\infty$ is also of Cauchy in $X$.
\end{Alem}

\begin{proof}
Let $\varepsilon>0$ be any given. Then by Theorem~\ref{Athm2}, one can find some number $\delta=\delta(\varepsilon/2)>0$ so that for any $x,z\in\textrm{cls}_X{G(y)}$ with $d(x,z)<\delta$ there follows
\begin{equation*}
d(t(x),t(z))<\frac{\varepsilon}{2}\quad \forall t\in G.
\end{equation*}
Now according to the hypotheses of the lemma, there exists an $N>0$ such that for any $n\ge N$ and $m\ge N$ we have
$$
d(t_n(y), t_m(y))<\delta\quad \textrm{and}\quad d(s_n(y), s_m(y))<\delta.
$$
From these inequalities we can get
$$
d(s_n\circ t_n(y), s_n\circ t_m(y))<\frac{\varepsilon}{2},\quad d(t_m\circ s_n(y), t_m\circ s_m(y))<\frac{\varepsilon}{2},
$$
and hence by the commutativity of $G$, we can obtain that
$$
d(t_n\circ s_n(y), t_m\circ s_m(y))<\varepsilon.
$$
This completes the proof of Lemma~\ref{Alem3}.
\end{proof}

The structure of a set consisting of a Bohr almost periodic motion is characterized by the following theorem, the $\mathbb{R}$-action case is due to V.V.~Nemytskii (cf.~\cite[Theorem~V8.16]{NS}).

\begin{Athm}\label{Athm4}%%%
Let $G\curvearrowright X$ be a continuous action of a topological abelian semigroup $G$ on the compact metric space $X$.
If $y\in X$ is a Bohr almost periodic point for $G\curvearrowright X$, then $\mathrm{cls}_X{G(y)}$ is a compact abelian semigroup having a binary operation $\diamond$ with
\begin{equation*}
g(y)\diamond h(y)=(g\circ h)(y)\quad \forall g,h\in G,
\end{equation*}
such that if $G$ has an identity $e$, then $\mathrm{cls}_X{G(y)}$ has the identity $y$ and that $G\curvearrowright\mathrm{cls}_X{G(y)}$ is characterized by the translation $(g,x)\mapsto g(x)=g(y)\diamond x$ for all $g\in G$ and $x\in\mathrm{cls}_X{G(y)}$.
\end{Athm}
\begin{rem}
See \cite{ST} in the case $G=\mathbb{R}$ and see \cite[Theorem~4.48]{GH} in the abelian group case for the related results.
\end{rem}

\begin{proof}
Let $y$ be Bohr almost periodic for $G\curvearrowright X$ and simply write $K=\textrm{cls}_X{G(y)}$. We will define in $K$ a commutative binary operation $\diamond$ as follows:

First, let $x,z\in G(y)$, i.e., $x=t_x(y)$ and $z=t_z(y)$ for some $t_x,t_z\in G$; the identity of the semigroup $K$ will be the point $y=e(y)$ if $G$ contains an identity $e$; we then define the commutative binary operation as $x\diamond z=t_x\circ t_z(y)=z\diamond x$. If $x=g(y)=g^\prime(y)$ for some pair $g,g^\prime\in G$ with $g\not=g^\prime$ and $z=t_z(y)$, then
$$
g\circ t_z(y)=t_z(g(y))=t_z(g^\prime(y))=g^\prime\circ t_z(y).
$$
Thus $x\diamond z$ is well defined and commutative in $G(y)$.
This binary operation $\diamond$ in $G(y)$ clearly satisfies the semigroup axioms and it is continuous.

We now need to extend this operation $\diamond$ by continuity to the whole of $K$. For this, let $x\in K$ with $x=\lim_{n\to\infty}t_n^x(y)$ and let $z\in K$ with $z=\lim_{n\to\infty}t_n^z(y)$. Then, by definition,
$$
x\diamond z=\lim_{n\to\infty}t_n^x\circ t_n^z(y)=z\diamond x.
$$
The above limit exists, since by Lemma~\ref{Alem3} the sequence $t_n^x\circ t_n^z(y)$ is of Cauchy in $X$ and $X$ is complete. Clearly this binary operation satisfies the algebraic axioms required by a semigroup.

Now we shall prove the continuity of the operation $\diamond$ defined above. For this, we let $x=\lim_{n\to\infty}t_n^x(y)$ and $z=\lim_{n\to\infty}t_n^z(y)$, and let there be given any $\varepsilon>0$. We define $\delta=\delta(\varepsilon/3)$ by Theorem~\ref{Athm2}. Assume
$$
d(x,x^\prime)<\frac{\delta}{3},\ x^\prime=\lim_{n\to\infty}t_n^{x^\prime}(y)\quad \textrm{and}\quad d(z,z^\prime)<\frac{\delta}{3},\ z^\prime=\lim_{n\to\infty}t_n^{z^\prime}(y).
$$
There is some $N>0$ such that for any $n\ge N$, we have
$$
d(x,t_n^x(y))<\frac{\delta}{3},\ d(x^\prime,t_n^{x^\prime}(y))<\frac{\delta}{3},
$$
and
$$
d(z,t_n^z(y))<\frac{\delta}{3},\ d(z^\prime,t_n^{z^\prime}(y))<\frac{\delta}{3}.
$$
Then we get
$$
d(t_n^x(y),t_n^{x^\prime}(y))<\delta\ \textrm{ and }\ d(t_n^z(y),t_n^{z^\prime}(y))<\delta\quad \forall n\ge N.
$$
Further, by the triangle inequality and the equicontinuity, we get
\begin{equation*}\begin{split}
d(x\diamond z,x^\prime\diamond z^\prime)&=d(x\diamond z,x^\prime\diamond z)+d(x^\prime\diamond z,x^\prime\diamond z^\prime)\\
&=\lim_{n\to\infty}d(t_n^x\circ t_n^z(y), t_n^{x^\prime}\circ t_n^z(y))+\lim_{n\to\infty}d(t_n^{x^\prime}\circ t_n^z(y), t_n^{x^\prime}\circ t_n^{z^\prime}(y))\\
&\le\frac{\varepsilon}{3}+\frac{\varepsilon}{3}
\end{split}\end{equation*}
as $n\ge N$. Therefore, under this binary operation $\diamond$, $K$ is a compact abelian semigroup with the required properties.

This completes the proof of Theorem~\ref{Athm4}.
\end{proof}

The above proof is an improvement of the necessity of \cite[Theorem~V.8.16]{NS} for $G=\mathbb{R}$. When $G=\mathbb{Z}$, then $G\curvearrowright K$ in Theorem~\ref{Athm4} is exactly a Kronecker system; cf.~\cite[Theorem~1.9]{Fur}.

In fact, we note here that if $G\curvearrowright X$ is topologically transitive and equicontinuous, then the statement of Theorem~\ref{Athm4} still holds by analogous arguments. In addition, if $G$ is assumed to be a topological abelian group, one can further show that $\mathrm{cls}_X{G(y)}$ is a compact abelian group.

%%%%%%%%%%%%%%%%%%%%%%%%%%%%%%%%%%%%%%%%%%%%%%%%%%%%
\section{Probabilistic structures}\label{sec3}%%%
We now turn to the probabilistic or ergodic theory of Bohr almost periodic motions of topological abelian semigroups acting continuously on compact metric spaces in the last two sections.

The classical Haar measure on a locally compact second countable group possesses the property that it has positive measures for every open subsets of the group. However, even for a compact countable abelian semigroup, this is not a case. For example, let $\bar{\mathbb{Z}}_+=\mathbb{Z}_+\cup\{\infty\}$ be the one-point compactification of the discrete additive semigroup $(\mathbb{Z}_+,+)$, endowed with the multiplicative binary operation $\circ$ as follows:
\begin{equation*}
s\circ t=s+t\quad \forall s,t\in\bar{\mathbb{Z}}_+.
\end{equation*}
Then the atomic probability measure $\delta_{\infty}$ concentrated at the point $\infty$ is the unique translation-invariant probability measure on $(\bar{\mathbb{Z}}_+,\circ)$.

As mentioned before, the classical Haar-Weil theorem that asserts the existence and uniqueness of Haar measures for locally compact second countable (abbreviated lcsc) groups, does not work in our situations. However we will need the following preliminary lemma.

\begin{Alem}[Haar measure]\label{Alem5}%%%
Let $(K,\diamond)$ be a compact second countable abelian \textit{semigroup}. Then $K\curvearrowright K$, given by $(g,k)\mapsto g(k)=g\diamond k$ for all $g,k\in K$, has a unique $K$-invariant probability measure $m_K$, called the \textit{Haar measure} of $K$ as in the lcsc group case.
\end{Alem}

\begin{proof}
\textit{Existence of Haar measure}: First of all, by the Markov-Kakutani fixed-point theorem, there exists at least one common invariant Borel probability measure, write $m_K$, on $K$ for all the commuting continuous transformations
\begin{equation*}
\left\{L_g\colon K\rightarrow K;\quad k\mapsto g(k)\right\}_{g\in K}.
\end{equation*}
Since $K$ is commutative, clearly $m_K$ is both left- and right-invariant for all $g\in K$; in other words, 
\begin{gather*}
m_K\circ L_g^{-1}=m_K=m_K\circ R_g^{-1}\quad \forall g\in G.
\end{gather*}

\textit{Unicity of Haar measure}: Let there exists another such Borel probability measure $\mu$ on $K$. Using the invariance of $m_K$ and $\mu$ together with Fubini's theorem, we get for any $\varphi\in C(K)$
\begin{equation*}\begin{split}
\int_K\varphi d\mu&=\int_K\left(\int_K\varphi(y)d\mu(y)\right)dm_K(x)\\
&=\int_K\left(\int_K\varphi(R_x(y))d\mu(y)\right)dm_K(x)\\
&=\int_K\left(\int_K\varphi(L_y(x))dm_K(x)\right)d\mu(y)\\
&=\int_K\left(\int_K\varphi(x)dm_K(x)\right)d\mu(y)\\
&=\int_K\varphi dm_K
\end{split}\end{equation*}
Since $\varphi$ is arbitrary, we conclude that $m_K=\mu$ and the asserted uniqueness follows.

This thus completes the proof of Lemma~\ref{Alem5}.
\end{proof}

The existence proofs of the classical Haar-Weil theorem presented in available literature for an lcsc group $G$ are complicated and the inversion $g\rightarrow g^{-1}$ from $G$ onto $G$ plays an important role (cf.~\cite[Theorem~14.14]{Roy}). To get around the difficulty caused by no inversion, we have employed the Markov-Kakutani theorem in the above proof of Lemma~\ref{Alem5}. In fact, this is not a new method for proving the existence of invariant measures. The credit goes back to at least Mahlon Day (1961).

Recall that for the $\mathbb{Z}_+$-action dynamical system $f\colon\mathbb{Z}_+\times X\rightarrow X$ on the compact metric space $X$, it is called \textit{strictly ergodic} if it consists of a unique ergodic set and all points of the underlying space $X$ are density points with respect to this measure; see \cite[Definition~VI.9.33]{NS}.
It is a well-known interesting fact that every minimal set consisting of Bohr almost periodic motions $f(\centerdot,x)$ is strictly ergodic (cf.~\cite[Theorem~VI.9.34]{NS}). To prove this one needs the equicontinuity of a Bohr almost periodic motion $f(\centerdot,x)$ and an ergodic theorem of Bohr that says the time-mean value
\begin{equation*}
\lim_{N\to\infty}\frac{1}{N}\sum_{t=0}^{N-1}\varphi(f(t,x))
\end{equation*}
exists for every $\varphi\in C(X)$.

Although there is no Bohr theorem at hands here, yet we can extend this to abelian semigroups in another way as follows:

\begin{Athm}\label{Athm6}%%%
Let $G\curvearrowright X$ be a continuous action of a topological abelian semigroup $G$ on the compact metric space $X$. If $y$ is a Bohr almost periodic point with $X=\mathrm{cls}_XG(y)$, then $G\curvearrowright X$ is \textit{uniquely ergodic}; i.e., it consists of a unique ergodic set (and all points of $X$ are density with respect to this measure if $G$ has inversion).
In particular, if additionally $G=\mathbb{Z}_+^d$ or $\mathbb{R}_+^d$, then $G\curvearrowright X$ is strictly ergodic.
\end{Athm}

\begin{proof}
Let $\mathscr{M}_{\textsl{inv}}(G\curvearrowright X)$ denote the weak-* compact convex set of all the $G$-invariant Borel probability measures on $X$. Since $G$ is commutative, it follows from the Markov-Kakutani fixed-point theorem that $\mathscr{M}_{\textsl{inv}}(G\curvearrowright X)\not=\varnothing$. It is easy to see that $\mu\in\mathscr{M}_{\textsl{inv}}(G\curvearrowright X)$ is ergodic if and only if it is an extremal point of the convex set $\mathscr{M}_{\textsl{inv}}(G\curvearrowright X)$ (see, e.g., \cite[Proposition~3.4]{Fur}).

Next, let $\mu\in\mathscr{M}_{\textsl{inv}}(G\curvearrowright X)$ be arbitrarily given. According to Theorem~\ref{Athm4}, for every $g\in G$, the left-translation
$$
L_g\colon X\rightarrow X;\quad x\mapsto g(y)\diamond x
$$
preserves the measure $\mu$ invariant, where $\diamond$ is the commutative multiplication in $X$ defined by Theorem~\ref{Athm4}.

Since $G(y)$ is dense in $X$, we can get that for every $z\in X$, the left-translation
$L_z\colon X\rightarrow X$, given by $x\mapsto z\diamond x$, preserves $\mu$ invariant as well; this is because as $g_n(y)\to z$ there follows that $\varphi(L_{g_n}(x))\to\varphi(L_z(x))$ for any $\varphi\in C(X)$. Therefore by Lemma~\ref{Alem5}, $\mu$ is just the Haar probability measure $m_X$ of $(X,\diamond)$, which is left and right invariant since $X$ is commutative and compact. This means that $\mathscr{M}_{\textsl{inv}}(G\curvearrowright X)$ exactly consists of a single point $\mu$ (such that all points of $X$ are density with respect to this measure if $G$ has inversion).

In particular, if $G=\mathbb{Z}_+^d$ or $\mathbb{R}_+^d$, then from \cite[Corollary~2.7]{ChD} it follows that $X$ is $G$-minimal and hence $G\curvearrowright X$ is strictly ergodic.

This completes the proof of Theorem~\ref{Athm6}.
\end{proof}

%%%%%%%%%%%%%%%%%%%%%%%%%%%%%%%%%%%%%%%%%%%%%%
\section{Bohr pointwise convergence theorem}\label{sec4}%%%
Now conversely our Theorem~\ref{Athm6} results in Bohr's pointwise convergence theorem for abelian semigroups acting continuously on a compact metric space.

From now on, let $(G,\circ)$ be an lcsc semigroup, where $\circ$ denotes the binary operation in $G$. By a \textit{Radon measure} on $G$ we mean here a Borel measure that is finite on each compact subset and positive on some compact subset of $G$.

The following concept is a generalization of the classical Haar measure of lcsc groups.

\begin{Adefn}\label{def7}%%%
A Radon measure $\lambda_G$ on $G$ is called a (left) \textit{quasi-Haar measure} of $G$ (for discrete $G$, we take this to be the counting measure $\#$ on $G$), if for any compact subsets $K$ of $G$ there holds $\lambda_G(K)=\lambda_G\left(g\circ K)\right)$ for any $g\in K$.
\end{Adefn}

For example, the standard Lebesgue measure on $G=\mathbb{R}_+^d$ is a left and right quasi-Haar measure, but not a Haar measure. Clearly, the counting measure $\#$ on $G=\mathbb{Z}_+^d$ is a quasi-Haar, but not Haar, measure. Both $\mathbb{R}_+^d$ and $\mathbb{Z}_+^d$ are such that the left translation $L_g\colon G\rightarrow G$ is continuous injective for each $g\in G$. In fact, if $G$ is discrete and if $L_g\colon t\mapsto g\circ t$ is continuous and injective for all $g\in G$, then $\lambda_G=\#$ is a left quasi-Haar measure on $G$.

We consider another lcsc discrete semigroup with no inversion. Let $G$ be the set of all nonsingular, nonnegative, and integer $n\times n$ matrices. Since $G$ may be thought of as an open subspace of $\mathbb{R}_+^{n\times n}$, it is an lcsc discrete semigroup under the standard matrix multiplication such that $L_g\colon G\rightarrow G$ is continuous and injective for each $g\in G$. Since the inverse of a nonnegative nonsingular matrix is not necessarily nonnegative, for example,
\begin{equation*}
\left[\begin{matrix}1&1\\0&1\end{matrix}\right]^{-1}=\left[\begin{matrix}1&-1\\0&1\end{matrix}\right],
\end{equation*}
$G$ is only a semigroup, but not a group, with the standard matrix multiplication operation.

\begin{Adefn}\label{def8}%%%
Let $\lambda_G$ be a left quasi-Haar measure on $G$. We refer to a sequence of compact subsets $\mathcal{F}=(F_n)_{n=1}^\infty$ of $G$ as a \textit{F{\o}lner sequence} w.r.t. $\lambda_G$, if
\begin{equation*}
\lim_{n\to\infty}\frac{\lambda_G\left(F_n\vartriangle g\circ F_n\right)}{\lambda_G(F_n)}=0\quad \forall g\in G.
\end{equation*}
Here $\vartriangle$ means the symmetric difference of sets.
\end{Adefn}

It is well known that if $G$ is an \textsl{amenable lcsc} group, then one can always find a F{\o}lner sequence $\mathcal{F}=(F_n)_1^\infty$ of compact subsets of $G$ w.r.t. the left Haar measure $\lambda_G$.
An abelian lcsc group is amenable.

For any Borel probability measure $\mu$ in the compact metric space $X$, write $\mathscr{C}_\textsl{b}^\mu(X)$ as the set of all bounded Borel real/complex functions defined on the space $X$ whose discontinuities form a set of $\mu$-measure zero.

\begin{Athm}\label{Athm9}%%%
Let $G$ be an lcsc abelian semigroup with a quasi-Haar measure $\lambda_G$ such that $L_g\colon G\rightarrow G$ is injective for each $g\in G$. Let $G\curvearrowright X$ be a Bohr almost periodic, continuous action of  $G$ on the compact metric space $X$, which preserves a Borel probability measure $\mu$ in $X$. Then for any F{\o}lner sequence $\mathcal{F}$ w.r.t. $\lambda_G$ of $G$ and any $\varphi\in \mathscr{C}_\textsl{b}^\mu(X)$,
\begin{equation*}
\frac{1}{\lambda_G(F_n)}\int_{F_n}\varphi(g(x))d\lambda_G(g)\to\int_X\varphi d\mu\quad \textrm{as }n\to\infty,
\end{equation*}
uniformly for $x\in\mathrm{supp}(\mu)$.
\end{Athm}

\begin{proof}
First of all we may assume that $X=\mathrm{supp}(\mu)$. Let $\mathcal{F}=(F_n)_{n=1}^\infty$ be an arbitrary F{\o}lner sequence of $G$ with respect to the quasi-Haar measure $\lambda_G$ and let $\varphi\in \mathscr{C}_\textsl{b}^\mu(X)$ be any given.

For any point $x\in X$ and $n\ge1$, using the Riesz representation theorem we now first define an empirical probability measure $\mu_{x,n}$ in $X$ as follows:
$$
\int_X\psi(y)d\mu_{x,n}(y)=\frac{1}{\lambda_G(F_n)}\int_{F_n}\psi(t(x))d\lambda_G(t),\quad \forall\psi\in C(X).
$$
Let $\tilde{\mu}$ be an arbitrary limit point of the sequence $(\mu_{x,n})_{n=1}^\infty$ under the weak-* topology; then from the basic property of F{\o}lner sequence it follows that $\tilde{\mu}$ is $G$-invariant. Thus $\tilde{\mu}=\mu$ by Theorem~\ref{Athm6} and then $\mu_{x,n}$ converges weakly-* to $\mu$ for all $x\in X$. This implies by \cite[Lemma~2.1]{Dai} that
\begin{equation*}
\frac{1}{\lambda_G(F_n)}\int_{F_n}\varphi(t(x))d\lambda_G(t)\to\int_X\varphi d\mu\quad \textrm{as }n\to\infty
\end{equation*}
for each point $x\in X$.

To prove the desired uniformity, we may assume $\int_X\varphi d\mu=0$ without loss of generality. By contradiction, let there exist some $\varepsilon>0$ and a sequence of points $(x_{n_k})_{k=1}^\infty$ in $X$ so that
\begin{equation*}
\left|\frac{1}{\lambda_G(F_{n_k})}\int_{F_{n_k}}\varphi(t(x_{n_k}))d\lambda_G(t)\right|\ge\varepsilon,\quad \forall k>0.
\end{equation*}
By throwing away a subsequence of $(n_k)$ if necessary, we can assume
$$
\mu_{x_{n_k},n_k}\xrightarrow[]{\textrm{weakly-*}}\mu\quad \textrm{as }k\to\infty.
$$
Then by \cite[Lemma~2.1]{Dai} once again
$$\frac{1}{\lambda_G(F_{n_k})}\int_{F_{n_k}}\varphi(t(x_{n_k}))d\lambda_G(t)\to0$$
which is a contradiction.

This concludes the proof of Theorem~\ref{Athm9}.
\end{proof}

We note that if the F{\o}lner sequence $\mathcal{F}$ of $G$ in Theorem~\ref{Athm9} satisfies the additional essential \textit{Shulman Condition}, that is, for some $C>0$ and all $n>0$ we have
\begin{equation*}
\lambda_G\left({\bigcup}_{k<n}F_k^{-1}F_n\right)\le C\lambda_G(F_n),
\end{equation*}
where $G$ need to be required to be a \textit{group} and $F_k^{-1}=\{g^{-1}\colon g\in F_k\}$,
then from the pointwise ergodic theorem of Lindenstrauss~\cite[Theorem~1.2]{Lin} it follows that the pointwise convergence holds for any $\varphi\in L^1(X,\mathscr{B}(X),\mu)$. Because of lacking of this Shulman condition in our context we, however, cannot generalize the pointwise convergence in Theorem~\ref{Athm9} to $L^1$-functions, not even for $L^\infty$-functions. Indeed for $G=\mathbb{Z}$, let $\mathcal{F}$ be chosen as
$$
F_n=\left\{n^2, n^2+1, \dotsc,n^2+n\right\};
$$
then A.~del Junco and J.~Rosenblatt showed in \cite{JR} that there always exists certain $\varphi\in L^\infty(X,\mathscr{B}(X),\mu)$ such that
$$\frac{1}{\#F_n}\sum_{g\in F_n}\varphi(g(x))$$
does not have a limit almost everywhere, if $(X,\mathscr{B}(X),\mu)$ is nontrivial.

%%%%%%%%%%%%%%%%%%%%%%%%%%%%%%%%%%%%
\section*{\textbf{Acknowledgments}}%
%The authors would like to thank Professor Victor~Kozyakin for many helpful discussion.
%%Particularly he is grateful to the anonymous reviewers for their comments.
This work was partly supported by National Natural Science Foundation of China grant $\#$11431012, 11271183 and PAPD of Jiangsu Higher Education Institutions.
%%%%%%%%%%%%%%%%%%%%%%%%%%%

\end{document}